\documentclass[10pt,a4paper,oneside]{article}
\usepackage[english]{babel}
\usepackage{a4wide,amsmath,amsthm,amssymb,url,graphicx,xspace,algorithm,algorithmic,pgf}
\usepackage[latin1]{inputenc}
\usepackage{enumitem}
\usepackage{multirow}
\usepackage{hyperref}
\usepackage{tabularx}

\def\Z{\mathbb{Z}}
\def\N{\mathbb{N}}

\DeclareMathOperator{\PG}{PG}

\theoremstyle{definition}
\newtheorem{theorem}{Theorem}[section]
\newtheorem{lemma}[theorem]{Lemma}
\newtheorem{definition}[theorem]{Definition}
\newtheorem{remark}[theorem]{Remark}
\newtheorem{corollary}[theorem]{Corollary}
\newtheorem{notation}[theorem]{Notation}
\newtheorem{example}[theorem]{Example}

\newcommand{\comments}[1]{}

\author{Maarten De Boeck\footnote{Address: UGent, Department of Mathematics, Krijgslaan 281-S22, 9000 Gent, Flanders, Belgium. \newline Email address: mdeboeck@cage.ugent.be}}
\title{The largest Erd\H{o}s-Ko-Rado sets in $2-(v,k,1)$ designs}
\date{}
\begin{document}
\maketitle

\begin{abstract}
  An Erd\H{o}s-Ko-Rado set in a block design is a set of pairwise intersecting blocks. In this article we study Erd\H{o}s-Ko-Rado sets in $2-(v,k,1)$ designs, Steiner systems. The Steiner triple systems and other special classes are treated separately. For $k\geq4$, we prove that the largest Erd\H{o}s-Ko-Rado sets cannot be larger than a point-pencil if $r\geq k^{2}-3k+\frac{3}{4}\sqrt{k}+2$ and that the largest Erd\H{o}s-Ko-Rado sets are point-pencils if also $r\neq k^{2}-k+1$ and $(r,k)\neq(8,4)$. For unitals we also determine an upper bound on the size of the second-largest maximal Erd\H{o}s-Ko-Rado sets.
\end{abstract}

\section{Introduction}
\label{sec:introduction}


\subsection{Block designs}

\begin{definition}
  A $t-(v,k,\lambda)$ block design, $v>k>1$, $k\geq t\geq 1$, $\lambda>0$, is an incidence geometry $\mathcal{D}=(\mathcal{P},\mathcal{B},\mathcal{I})$ with incidence relation $\mathcal{I}$, such that $|\mathcal{P}|=v$, such that any element of $\mathcal{B}$ (blocks) is incident with $k$ elements of $\mathcal{P}$ (points) and such that any $t$ points are contained in $\lambda$ common blocks. A block can be identified with the $k$-subset of $\mathcal{P}$ which it determines.
\end{definition}

Block designs have been widely studied for many years, see for example \cite{ak,cvl,cd,dem,hp} for an overview.
\par The following counting results are widely known.

\begin{theorem}
  Let $\mathcal{D}=(\mathcal{P},\mathcal{B},\mathcal{I})$ be a $t-(v,k,\lambda)$ block design. Then,
  \begin{itemize}
    \item the number of blocks through an arbitrary set of $i$ points equals $\lambda_{i}=\lambda\binom{v-i}{t-i}/\binom{k-i}{t-i}$;
    \item in particular, the number of blocks through a fixed point equals $r=\lambda_{1}=\lambda\binom{v-1}{t-1}/\binom{k-1}{t-1}$;
    \item $b=|\mathcal{B}|=\frac{vr}{k}$.
  \end{itemize}
\end{theorem}

The most studied class of block designs are the $2-(v,k,1)$ designs, which are called Steiner systems. Among them we mention especially the $2-(n^{2}+n+1,n+1,1)$ designs (the projective planes of order $n$), $n\geq2$, the $2-(n^{2},n,1)$ designs (the affine planes of order $n$), $n\geq3$, and the $2-(n^{3}+1,n+1,1)$ designs (the unitals of order $n$), $n\geq2$.
\par By the above results, a $2-(v,k,1)$ design contains $b=\frac{v(v-1)}{k(k-1)}$ blocks, $r=\frac{v-1}{k-1}$ of them through a fixed point. Note that a $2-(v,k,1)$ design can only exist if $v\equiv 1\pmod{k-1}$ and $k(k-1)\mid v(v-1)$.

\subsection{\texorpdfstring{Erd\H{o}s}{Erdos}-Ko-Rado theorems}

In 1961, the original Erd\H{o}s-Ko-Rado theorem solved a problem in extremal combinatorics.

\begin{theorem}[\cite{ekr}]
  Let $\Omega$ be a set of size $n$ and $\mathcal{S}$ a family of subsets of size $k$ such that the elements of $\mathcal{S}$ are pairwise not disjoint. If $n\geq2k$, then $|\mathcal{S}|\leq\binom{n-1}{k-1}$. If $n\geq N_{0}(k)$, then equality holds if and only if $\mathcal{S}$ is the set of all subsets of size $k$ containing a fixed element of $\Omega$.
\end{theorem}

In 1984, Wilson showed that the bound $n\geq 2k+1$ is both sufficient and necessary for the above classification: the families $\mathcal{S}$ meeting the upper bound are sets of all subsets of size $k$ containing a fixed element of $\Omega$ (\cite{w}).

Many generalisations of this problem have been investigated. The set $\Omega$ has often been replaced by a geometry, such as a vector space or a polar space, simultaneously replacing the subsets by subspaces. In general, an {\em Erd\H{o}s-Ko-Rado set} is a set of subsets (subspaces of fixed dimension) which are pairwise non-disjoint. It is called {\em maximal} if it can not be extended to a larger Erd\H{o}s-Ko-Rado set. Hence, an Erd\H{o}s-Ko-Rado set on a design is a set of pairwise intersecting blocks. The Erd\H{o}s-Ko-Rado problem asks for the classification of the (largest) Erd\H{o}s-Ko-Rado sets.
\par In \cite[Section 2]{bbsw} and \cite{ds}, surveys of Erd\H{o}s-Ko-Rado theorems in geometrical settings can be found. Recent results on Erd\H{o}s-Ko-Rado sets in projective and polar spaces can be found in e.g. \cite{bbc,bbs,deb,im,phdnewman,psv,tan}.
\par The most important type of Erd\H{o}s-Ko-Rado sets are the sets of all subsets (blocks, subspaces, ...) through a fixed point. They are called {\em point-pencils}. In a block design a point-pencil is a maximal Erd\H{o}s-Ko-Rado set if $r>k$.
\par For general block designs, the following Erd\H{o}s-Ko-Rado result was obtained by Rands.

\begin{theorem}[\cite{ran}]
  Let $\mathcal{D}=(\mathcal{P},\mathcal{B},\mathcal{I})$ be a $t-(v,k,\lambda)$ block design and let $\mathcal{S}$ be a subset of $\mathcal{B}$ such that the blocks of $\mathcal{S}$ have pairwise at least $s$ points in common, $0<s<t\leq k$.
  \begin{itemize}
    \item If $s<t-1$ and $v\geq s+\binom{k}{s}(k-s+1)(k-s)$, or
    \item if $s=t-1$ and $v\geq s+\binom{k}{s}^{2}(k-s)$,
  \end{itemize}
  then $|\mathcal{S}|\leq\lambda_{s}$ and equality is obtained if and only if $\mathcal{S}$ is the set of blocks through $s$ fixed points.
\end{theorem}

For an Erd\H{o}s-Ko-Rado set in a $2-(v,k,1)$ design, this implies the following corollary.

\begin{corollary}\label{partialgeometry}
  Let $\mathcal{D}$ be a $2-(v,k,1)$ block design and let $\mathcal{S}$ be an Erd\H{o}s-Ko-Rado set of $\mathcal{D}$, $k\geq 2$. If $v\geq1+k^{2}(k-1)$, then $|\mathcal{S}|\leq r$ and $|\mathcal{S}|=r$ if and only if $\mathcal{S}$ is a point-pencil.
\end{corollary}

In the same article (\cite{ran}), it is claimed that the bound $v\geq1+k^{2}(k-1)$ can be improved to $v>k^{3}-2k^{2}+2k$, but there is no proof of this statement. However, it is shown that the bound $v> k^{3}-2k^{2}+2k$ is sharp. If $v=k^{3}-2k^{2}+2k$ and $k-1$ is a prime power, the $2-(v,k,1)$ design consisting of the points and lines of $\PG(3,k-1)$ contains two different types of Erd\H{o}s-Ko-Rado sets of size $r=k^{2}-k+1$: the set of all blocks through a fixed point and the set of blocks arising from the set of lines in a fixed plane.
\par In this article, we will prove the result about the bound $v>k^{3}-2k^{2}+2k$ (see Theorem \ref{rands}). It follows from two easy observations. The main part of this paper is devoted to the investigation of $2-(v,k,1)$ designs with $v<k^{3}-2k^{2}+2k$ (see Theorem \ref{maintheorem}). It turns out that $v=k^{3}-2k^{2}+2k$ is an isolated case. For $2-(v,k,1)$ designs with $v$ smaller than $k^{3}-2k^{2}+2k$ but not much, the largest Erd\H{o}s-Ko-Rado sets are also point-pencils. The results are summarized in Theorem \ref{class3} and Corollary \ref{overview}.



\section{Some special Steiner systems}

\begin{remark}
  Let $\mathcal{D}$ be a $2-(v,k,1)$ design. For every point $P$ in $\mathcal{D}$, there is a block not containing this point since $v>k$. Each of the points on this block determines a different block through $P$. Hence, $r\geq k$. If $r=k$, then $\mathcal{D}$ is a projective plane of order $k-1$; if $r=k+1$, then $\mathcal{D}$ is an affine plane of order $k$. So, the projective and affine planes are the two `smallest' $2-(v,k,1)$ designs.
\end{remark}

We look at the projective and affine planes is detail.

\begin{remark}
  In a projective plane, every two blocks have a point in common. Hence, in a projective plane there is only one maximal Erd\H{o}s-Ko-Rado set of blocks, namely the set of all blocks. Recall that we mentioned in the introduction that a point-pencil is only maximal if $r>k$.
\end{remark}

\begin{remark}\label{affine}
  In an affine plane of order $n$, the set of blocks can be partitioned in $n+1$ classes of $n$ blocks, such that the blocks in the same class pairwise have no point in common. These are commonly called parallel classes. Two blocks of different classes always meet in a point. An Erd\H{o}s-Ko-Rado set contains necessarily at most one block of each parallel class. A maximal Erd\H{o}s-Ko-Rado set contains precisely one block of each parallel class. Consequently, every maximal Erd\H{o}s-Ko-Rado set contains $n+1$ blocks.
  \par It should be noted that not all these maximal Erd\H{o}s-Ko-Rado sets are isomorphic. Also note that the point-pencil can be described in this way.
\end{remark}

Now we turn our attention to $2-(v,k,1)$ designs with a special property.

\begin{definition}
  The {\em O'Nan configuration} in a design $\mathcal{D}$ is a set of four blocks, pairwise non-disjoint, such that no three contain a common point.
\end{definition}

We will show that we can find a complete classification of the maximal Erd\H{o}s-Ko-Rado sets on designs not containing an O'Nan configuration. Note that all projective planes and all affine planes of order at least $3$ do contain O'Nan configurations.

\par We already know the point-pencil, a maximal Erd\H{o}s-Ko-Rado set of size $r$. We now give an example of a maximal Erd\H{o}s-Ko-Rado set on a design without an O'Nan configuration.

\begin{example}
  Let $\mathcal{D}$ be a $2-(v,k,1)$ design without an O'Nan configuration. Let $P$ be a point and let $B$ be a block of $\mathcal{D}$ such that $P\notin B$. Let $\mathcal{S}$ be the union of $\{B\}$ and the set of all blocks through $P$ meeting $B$. It is obvious that all blocks of $\mathcal{S}$ meet each other, hence that $\mathcal{S}$ is an Erd\H{o}s-Ko-Rado set. We call it the {\em triangle}. It contains $k+1$ blocks. We prove that it is maximal.
  \par Let $L$ be a block of $\mathcal{D}$ not in $\mathcal{S}$, meeting all blocks of $\mathcal{S}$. The block $L$ cannot pass through $P$, hence meets all blocks of $\mathcal{S}$ through $P$ in a different point. Since $L\neq B$, we know $k\geq3$. Let $P'$ and $P''$ be two points on $B\setminus\{L\cap B\}$ and let $B'$ and $B''$ be the blocks of $\mathcal{S}$ through $P$, respectively meeting $B$ in the points $P'$ and $P''$. Then the blocks $B$, $L$, $B'$ and $B''$ determine an O'Nan configuration, a contradiction.
\end{example}

\begin{theorem}\label{classONan}
  Let $\mathcal{D}$ be a $2-(v,k,1)$ design without an O'Nan configuration and let $\mathcal{S}$ be a maximal Erd\H{o}s-Ko-Rado set on $\mathcal{D}$. Then, $\mathcal{S}$ is a point-pencil or a triangle.
\end{theorem}
\begin{proof}
  Assume that $\mathcal{S}$ is not a point-pencil; then we can find three blocks in $\mathcal{S}$, say $B_{1}$, $B_{2}$ and $B_{3}$, not through a common point. Denote $P_{1}=B_{2}\cap B_{3}$, $P_{2}=B_{3}\cap B_{1}$ and $P_{3}=B_{1}\cap B_{2}$. Any block $B\in\mathcal{S}$ should have a non-empty intersection with as well $B_{1}$, $B_{2}$ as $B_{3}$. Since $\mathcal{D}$ does not contain an O'Nan configuration, $B$ must pass through $P_{1}$, $P_{2}$ or $P_{3}$.
  \par If the block $B'_{i}\in\mathcal{S}$ passes through $P_{i}$, $B'_{i}\notin\{B_{1},B_{2},B_{3}\}$, and the block $B'_{j}\in\mathcal{S}$ passes through $P_{j}$, $B'_{j}\notin\{B_{1},B_{2},B_{3}\}$, $1\leq i\neq j\leq 3$, then the blocks $B_{i}$, $B_{j}$, $B'_{i}$ and $B'_{j}$ determine an O'Nan configuration, a contradiction. Hence, all blocks of $\mathcal{S}\setminus\{B_{1},B_{2},B_{3}\}$ pass through the same point $P_{i}$, $1\leq i\leq3$. Since $\mathcal{S}$ is maximal, it has to be a triangle based on the point $P_{i}$ and the block $B_{i}$.
\end{proof}

Note that $r>k+1$ for all $2-(v,k,1)$ designs without an O'Nan configuration, but the affine plane of order $2$. Hence, the point-pencil is the largest Erd\H{o}s-Ko-Rado set in these designs.
Of course, the above result only makes sense if $2-(v,k,1)$ designs without an O'Nan configuration exist. We give an example.

\begin{example}
  Let $\mathcal{H}(2,q^{2})$ be a non-singular Hermitian variety in $\PG(2,q^{2})$, the Desarguesian projective plane of order $q^{2}$. Up to projective transformations it is defined by $X^{q+1}_{0}+X^{q+1}_{1}+X^{q+1}_{2}=0$. The set of points on $\mathcal{H}(2,q^{2})$ and the secant lines to $\mathcal{H}(2,q^{2})$ in $\PG(2,q^{2})$, determine a unital. This unital is known as the {\em classical unital} or {\em Hermitian unital}.
\end{example}

\begin{theorem}[\cite{ona}]
  A classical unital $\mathcal{U}$ does not contain an O'Nan configuration.
\end{theorem}

It is conjectured that the classical unitals are the only unitals not containing an O'Nan configuration, see \cite{bro,pip}. In \cite{bro} this conjecture is proven to be true for unitals of order $3$. The unique unital of order $2$ is also classical.

\begin{corollary}\label{classclass}
  On a classical unital there are only two types of maximal Erd\H{o}s-Ko-Rado sets, the point-pencil and the triangle.
\end{corollary}

\section{The counting arguments}

In this section we will study maximal Erd\H{o}s-Ko-Rado sets in $2-(v,k,1)$ designs that are not point-pencils.

\begin{notation}\label{nota}
  Let $\mathcal{D}$ be a $2-(v,k,1)$ design and let $\mathcal{S}$ be an Erd\H{o}s-Ko-Rado set on $\mathcal{D}$. Denote the set of points of $\mathcal{D}$ covered by the blocks of $\mathcal{S}$ by $\mathcal{P}'$.
  \par We denote the number of points of $\mathcal{P}'$ that are contained in precisely $i$ blocks of $\mathcal{S}$ by $k_{i}$. Furthermore we denote $k_{\mathcal{S}}=\max\{i\mid k_{i}>0\}$.
\end{notation}

\begin{lemma}\label{upperbounds}
  Let $\mathcal{D}$ be a $2-(v,k,1)$ design and let $\mathcal{S}$ be an Erd\H{o}s-Ko-Rado set on $\mathcal{D}$. Then $|\mathcal{S}|\leq k_{\mathcal{S}}k-k+1$. If $\mathcal{S}$ is maximal and different from the point-pencil, then $k_{\mathcal{S}}\leq k$.
\end{lemma}
\begin{proof}
  Fix a block $C\in\mathcal{S}$. All blocks of $\mathcal{S}$ have a nontrivial intersection with $C$, so
  \[
    |\mathcal{S}|\leq 1+k(k_{\mathcal{S}}-1)=k_{\mathcal{S}}k-k+1\;.
  \]
  Now we prove the second part of the lemma. For every point $P\in\mathcal{P}'$, we can find a block $B\in\mathcal{S}$ not passing through $P$, since $\mathcal{S}$ is maximal but not a point-pencil. Any block of $\mathcal{S}$ through $P$ should meet $B$ and there is at most one block in $\mathcal{S}$ through $P$ and a given point of $B$. Hence, there are at most $k$ blocks in $\mathcal{S}$ passing through $P$. Consequently, $k_{\mathcal{S}}\leq k$.
\end{proof}

\begin{lemma}\label{maximalisatie}
  Choose $l\in\N\setminus\{0,1\}$, and $a,b\in\Z$ with
  \begin{align*}
    a&\geq\max\left\{-l(r-l-1)+1-\frac{br}{l+1},-\frac{b(b-1)}{(l+1)l}-2(b-1)\right\}\;,\\
    a&\leq\frac{rl-l^{2}+l-1}{l-1}-\frac{b(2l^{2}+2l-r+b-1)}{l^{2}-1}\;.
  \end{align*}
  Let $n_{1},\dots,n_{l}\in\N$ be such that $\sum^{l}_{i=1}in_{i}=(a-1)(l+1)+br+l(l+1)(r-l-1)$ and $\sum^{l}_{i=2}i(i-1)n_{i}=b(b-1)+l(l+1)(a+2b-2)$. Then $\sum^{l}_{i=2}(i-1)n_{i}\leq\binom{b}{2}+(a+2b-2)\binom{l+1}{2}$.
\end{lemma}
\begin{proof}
  Note that the inequalities $-l(r-l-1)+1-\frac{br}{l+1}\leq a$ and $-\frac{b(b-1)}{(l+1)l}-2(b-1)\leq a$ are present to ensure that both $a(l+1)+br+l(l+1)(r-l-1)-l-1$ and $b(b-1)+l(l+1)(a+2b-2)$ are nonnegative.
  \par Using the first equality, we can express $n_{1}$ as a function of $l$, $a$, $b$ and $n_{2},\dots,n_{l}$. Note that
  \begin{align*}
    n_{1}&=(a-1)(l+1)+br+l(l+1)(r-l-1)-\sum^{l}_{i=2}in_{i}\\
    &\geq(a-1)(l+1)+br+l(l+1)(r-l-1)-\sum^{l}_{i=2}i(i-1)n_{i}\\
    &=(a-1)(l+1)+br+l(l+1)(r-l-1)-b(b-1)-l(l+1)(a+2b-2)\\
    &=-a(l^{2}-1)-b(b-1)-b(2l^{2}+2l-r)+(l+1)(rl-l^{2}+l-1)\\
    &\geq 0\;,
  \end{align*}
  by the assumption. Hence, for every choice of $n_{2},\dots,n_{l}$, we can find a value $n_{1}\in\N$ such that the first equality holds. Now, we focus on the second equality. Assume that $n_{j}>0$ for a value $j\geq 3$. Then define $n'_{j}=n_{j}-1$, $n'_{2}=n_{2}+\frac{j(j-1)}{2}$ and $n'_{k}=n_{k}$ for $k\notin\{2,j\}$. It follows that
  \[
    \sum^{l}_{i=2}i(i-1)n'_{i}=\sum^{l}_{i=2}i(i-1)n_{i}=b(b-1)+l(l+1)(a+2b-2)\;.
  \]
  However,
  \[
    \sum^{l}_{i=2}(i-1)n'_{i}=\left(\sum^{l}_{i=2}(i-1)n_{i}\right)-(j-1)+\frac{j(j-1)}{2}>\sum^{l}_{i=2}(i-1)n_{i}\;,
  \]
  since $j\geq3$. So, repeatedly applying the above construction, we find that $\sum^{l}_{i=2}(i-1)n_{i}$ is maximal if $n_{i}=0$ for all $i\geq3$ and $n_{2}=\binom{b}{2}+(a+2b-2)\binom{l+1}{2}$. The lemma follows.
\end{proof}

\begin{lemma}\label{mainlemma}
  Let $\mathcal{D}$ be a $2-(v,k,1)$ design with replication number $r=\frac{v-1}{k-1}$, $k\geq3$, and let $\mathcal{S}$ be an Erd\H{o}s-Ko-Rado set on $\mathcal{D}$ 
  such that $|\mathcal{P}'|=k(k-1)+b$. Then
  \begin{multline*}
    |\mathcal{S}|\leq \max\left\{k^{2}-k+1-2\frac{(r-k)(k^{2}-k+1-r)}{k(k-2)}+\frac{b(b-1)}{(k-1)(k-2)}+2\frac{(b-1)(k^{2}-k-r)}{(k-1)(k-2)},\right.\\\left.k^{2}-r-\frac{r-1}{k-2}+\frac{b(b-1-r+2k(k-1))}{k(k-2)}\right\}\;.
  \end{multline*}
\end{lemma}
\begin{proof}
  Recall that $\mathcal{B}$ is the set of blocks of $\mathcal{D}$. We denote the subset of $\mathcal{B}$ containing precisely $i$ points of $\mathcal{P}'$ by $\mathcal{B}_{i}$ and we also denote $m_{i}=|\mathcal{B}_{i}|$. Note that $\mathcal{S}\subseteq\mathcal{B}_{k}$. We define $a:=k^{2}-k+1-|\mathcal{B}_{k}|$.
  Counting the tuples $(P,B)$ with $P\in\mathcal{P}'$, $B\in\mathcal{B}$ and $P$ on $B$, we find
  \[
    \sum^{k}_{i=1}im_{i}=(k(k-1)+b)r\;.
  \]
  Now applying $m_{k}=k^{2}-k+1-a$, we find
  \[
    m_{1}=(k(k-1)+b)r-\sum^{k-1}_{i=2}im_{i}-k(k^{2}-k+1-a)=k(k-1)(r-k)+(a-1)k+br-\sum^{k-1}_{i=2}im_{i}\;.
  \]
  Counting the tuples $(P,P',B)$ with $P,P'\in\mathcal{P}'$, $B\in\mathcal{B}$, $P\neq P'$ and both $P$ and $P'$ on $B$, we find
  \[
    \sum^{k}_{i=1}i(i-1)m_{i}=(k(k-1)+b)(k(k-1)+b-1)\;.
  \]
  Hence,
  \[
    \sum^{k-1}_{i=2}i(i-1)m_{i}=(k(k-1)+b)(k(k-1)+b-1)-k(k-1)(k^{2}-k+1-a)=b(b-1)+(a+2b-2)k(k-1)\;.
  \]
  Now we consider the set $T$ of triples $(P,P',B)$ with $P,P'\in\mathcal{P}\setminus\mathcal{P}'$, $B\in\mathcal{B}_{1}$, $P,P'\in B$ and $P\neq P'$. On the one hand we know
  \begin{multline*}
    |T|=m_{1}(k-1)(k-2)\\=k(k-1)^{2}(k-2)(r-k)+(a-1)k(k-1)(k-2)+br(k-1)(k-2)-(k-1)(k-2)\sum^{k-1}_{i=2}im_{i}\;.
  \end{multline*}
  On the other hand, using $|\mathcal{P}\setminus\mathcal{P}'|=v-k(k-1)-b=(r-k)(k-1)-(b-1)$, we can also find that
  \[
    |T|\leq \left((r-k)(k-1)-(b-1)\right)\left((r-k)(k-1)-b\right)-\sum^{k-1}_{i=2}(k-i)(k-i-1)m_{i}\;.
  \]
  Comparing this equality and inequality for $|T|$, we find
  \begin{multline*}
    \sum^{k-1}_{i=2}(k(k-1)(i-1)-i(i-1))m_{i}\geq k(k-1)^{2}(k-2)(r-k)+(a-1)k(k-1)(k-2)\\+br(k-1)(k-2)-b(b-1)-(r-k)^{2}(k-1)^{2}+(2b-1)(r-k)(k-1)\;.
  \end{multline*}
  Using the formula for $\sum^{k-1}_{i=2}i(i-1)m_{i}$, and dividing both sides by $k-1$, it follows that
  \[
    k\sum^{k-1}_{i=2}(i-1)m_{i}\geq ak(k-1)+bkr-k^{2}+(r-k)(k^{3}-2k^{2}-(r-1)(k-1))\;.
  \]
  We distinguish between two cases. If $a>r-k+1+\frac{r-1}{k-2}-\frac{b(b-1-r+2k(k-1))}{k(k-2)}$, then $|\mathcal{S}|\leq|\mathcal{B}_{k}|\leq k^{2}-r-\frac{r-1}{k-2}+\frac{b(b-1-r+2k(k-1))}{k(k-2)}$. If $a\leq r-k+1+\frac{r-1}{k-2}-\frac{b(b-1-r+2k(k-1))}{k(k-2)}$, we can apply Lemma \ref{maximalisatie} with $l=k-1$. Note that the conditions $-l(r-l-1)+1-\frac{br}{l+1}\leq a$ and $-\frac{b(b-1)}{(l+1)l}-2(b-1)\leq a$ are fulfilled since $\sum^{k-1}_{i=2}i(i-1)m_{i}$ and $\sum^{k-1}_{i=1}im_{i}$ are nonnegative. We find
  \[
    k\binom{b}{2}+k(a+2b-2)\binom{k}{2}\geq ak(k-1)+bkr-k^{2}+(r-k)(k^{3}-2k^{2}-(r-1)(k-1))\;,
  \]
  hence
  \[
    a\geq\frac{2(r-k)(k^{2}-k+1-r)}{k(k-2)}-\frac{2(b-1)(k^{2}-k-r)}{(k-1)(k-2)}-\frac{b(b-1)}{(k-1)(k-2)}\;.
  \]
  We find thus that
  \[
    |\mathcal{S}|\leq|\mathcal{B}_{k}|\leq k^{2}-k+1-\frac{2(r-k)(k^{2}-k+1-r)}{k(k-2)}+\frac{2(b-1)(k^{2}-k-r)}{(k-1)(k-2)}+\frac{b(b-1)}{(k-1)(k-2)}\;,
  \]
  which finishes the proof.
\end{proof}

Using the substitution $R=(k-1)^{2}-r$, we can rewrite this lemma.

\begin{corollary}\label{mainlemmaR}
  Let $\mathcal{D}$ be a $2-(v,k,1)$ design, $k\geq3$, and denote $(k-1)^{2}-r=(k-1)^{2}-\frac{v-1}{k-1}$ by $R$. Let $\mathcal{S}$ be an Erd\H{o}s-Ko-Rado set on $\mathcal{D}$ 
  such that $|\mathcal{P}'|=k(k-1)+b$. Then
  \begin{multline*}
    |\mathcal{S}|\leq \max\left\{k^{2}-k+1-2\frac{(k^{2}-3k+1-R)(k+R)}{k(k-2)}+\frac{b(b-1)}{(k-1)(k-2)}+2\frac{(b-1)(k-1+R)}{(k-1)(k-2)},\right.\\\left.k-1+R+\frac{R}{k-2}+\frac{b(b+k^{2}+R-2)}{k(k-2)}\right\}\;.
  \end{multline*}
\end{corollary}

\begin{lemma}\label{kpunt}
  Let $\mathcal{D}$ be a $2-(v,k,1)$ design and let $\mathcal{S}$ be an Erd\H{o}s-Ko-Rado set on $\mathcal{D}$ 
  with $k_{\mathcal{S}}=k$. Then $|\mathcal{P}'|=k^{2}-k+1$.
\end{lemma}
\begin{proof}
  Since $k_{\mathcal{S}}=k$, we can find a point $P\in\mathcal{P}'$ lying on $k$ blocks of $\mathcal{S}$. Denote these blocks by $B_{1},\dots,B_{k}$ and denote the set of points covered by these blocks by $\mathcal{P}''$. Any block of $\mathcal{S}$ not through $P$ contains a point on each of the blocks $B_{i}$, $i=1,\dots,k$. Since a block contains precisely $k$ points, all points on such a block are contained in $\mathcal{P}''$. Hence, $\mathcal{P}''=\mathcal{P}'$ and
  \[
    |\mathcal{P}''|=|\bigcup^{k}_{i=1}B_{i}|=1+k(k-1)=k^{2}-k+1\;.\qedhere
  \]
\end{proof}


\begin{lemma}\label{k-1punt}
  Let $\mathcal{D}$ be a $2-(v,k,1)$ design and let $\mathcal{S}$ be an Erd\H{o}s-Ko-Rado set on $\mathcal{D}$ 
  with $k_{\mathcal{S}}=k-1$. Write $a'=(k-1)^{2}-|\mathcal{S}|$. If $a'<k-1$, then $k(k-1)\leq|\mathcal{P}'|\leq k(k-1)+\frac{a'^{2}-a'}{k-1-a'}$.
\end{lemma}
\begin{proof}
  First we will prove that there is a block in $\mathcal{S}$ containing at least two points that are on $k-1$ blocks of $\mathcal{S}$. Assume there is no such block and choose a block $C$. At most one point on $C$ belongs to $k-1$ blocks of $\mathcal{S}$. However, all blocks of $\mathcal{S}$ have a nontrivial intersection with $C$, so
  \[
    |\mathcal{S}|\leq 1+(k-2)+(k-1)(k-3)=(k-1)(k-2)\;,
  \]
  hence $a'\geq k-1$, which contradicts the assumption $a'<k-1$.
  \par Let $B_{1}$ be a block of $\mathcal{S}$ through the points $Q_{1}$ and $Q_{2}$, both on $k-1$ blocks of $\mathcal{S}$, and let $B_{1},B_{2},\dots,B_{k-1}$ and $B_{1}=C_{1},C_{2},\dots,C_{k-1}$ be the blocks of $\mathcal{S}$, respectively through $Q_{1}$ and $Q_{2}$. There are $(k-2)^{2}$ points which lie on a block $B_{j}$ and also on a block $C_{j'}$, $2\leq j,j'\leq k-1$; there are $k-2$ points which lie on a block $B_{i}$, but not on a block $C_{i'}$, and there are also $k-2$ points which lie on a block $C_{i}$, but not on a block $B_{i'}$; the block $B_{1}=C_{1}$ contains $k$ points. Hence, $|\mathcal{P}'|\geq (k-2)^{2}+2(k-2)+k=k(k-1)$.
  \par Now, recall the notation $k_{i}$. By standard counting arguments we know that
  \[
    \sum^{k-1}_{i=1}ik_{i}=((k-1)^{2}-a')k\quad\text{ and }\quad\sum^{k-1}_{i=1}i(i-1)k_{i}=((k-1)^{2}-a')(k(k-2)-a')\;.
  \]
  Let $j\in\N\setminus\{0\}$ be the smallest value such that $k_{j}\neq0$ and let $R$ be a point of $\mathcal{P}'$ on $j$ blocks of $\mathcal{S}$. Let $B\in\mathcal{S}$ be a block through $R$. All blocks of $\mathcal{S}$ meet $B$, hence
  \[
    |\mathcal{S}|=(k-1)^{2}-a'\leq 1+(k-1)(k-2)+(j-1)\;.
  \]
  It follows that $j\geq k-1-a'$. Therefore, the following inequality holds:
  \[
    \sum^{k-1}_{i=1}(i-(k-a'-1))(k-1-i)k_{i}\geq0\:.
  \]
  So,
  \begin{align*}
    0&\leq-\sum^{k-1}_{i=1}i(i-1)k_{i}+(2k-a'-3)\sum^{k-1}_{i=1}ik_{i}-(k-a'-1)(k-1)\sum^{k-1}_{i=1}k_{i}\\
    &=-((k-1)^{2}-a')(k(k-2)-a')+(2k-a'-3)((k-1)^{2}-a')k-(k-a'-1)(k-1)\sum^{k-1}_{i=1}k_{i}\\
    &=((k-1)^{2}-a')(k-1)(k-a')-(k-a'-1)(k-1)\sum^{k-1}_{i=1}k_{i}\;.
  \end{align*}
  Consequently,
  \[
    |\mathcal{P}'|=\sum^{k-1}_{i=1}k_{i}\leq\frac{((k-1)^{2}-a')(k-a')}{k-a'-1}=k(k-1)+\frac{a'^{2}-a'}{k-a'-1}\;
  \]
  and the lemma follows.
\end{proof}

\section{Classification results for \texorpdfstring{$k=3$}{k=3}}

For $k=2$, a $2-(v,k,1)$ design is a complete graph $K_{v}$ on $v$ vertices, the edges being the blocks. It can immediately be seen that there are precisely two different types of maximal Erd\H{o}s-Ko-Rado sets on $K_{v}$, namely the point-pencil, which contains $v-1$ blocks, and the triangle, a set $\{\{p_{1},p_{2}\},\{p_{1},p_{3}\},\{p_{2},p_{3}\}\}$ for three points $p_{1},p_{2},p_{3}\in\mathcal{P}$, which contains 3 blocks.
\par So, the first nontrivial case is $k=3$. A $2-(v,3,1)$ design is called a Steiner triple system of size $v$. Steiner triple systems exist if and only if $v\equiv1,3\pmod{6}$ and $v\geq7$. Up to isomorphism, there is only one Steiner triple system for $v=7$, namely the Fano plane, the projective plane of order $2$; there is only one Steiner triple system for $v=9$, namely the affine plane of order $3$; and there are two Steiner triple systems for $v=13$. For more details, we refer the interested reader to \cite[Section II.1, Section II.2]{cd}. 

\begin{theorem}
  Let $\mathcal{D}$ be a $2-(v,3,1)$ design and let $\mathcal{S}$ be a maximal Erd\H{o}s-Ko-Rado set of $\mathcal{D}$. Then $\mathcal{S}$ belongs to one of five types. The maximal Erd\H{o}s-Ko-Rado sets contain $\frac{v-1}{2}$, $4$, $5$, $6$ or $7$ blocks. Each type corresponds to a size and vice versa.
\end{theorem}
\begin{proof}
  If all blocks of $\mathcal{S}$ pass through a common point, then $\mathcal{S}$ is a point-pencil and it contains $\frac{v-1}{2}$ blocks. So, from now on we assume that there is no point on all blocks of $\mathcal{S}$. Let $B_{1},B_{2},B_{3}\in\mathcal{S}$ be three blocks such that $B_{1}\cap B_{2}=\{P_{3}\}$, $B_{1}\cap B_{3}=\{P_{2}\}$ and $B_{2}\cap B_{3}=\{P_{1}\}$, with $P_{1}, P_{2},P_{3}$ three different points. Let $Q_{i}$ be the third point on the block $B_{i}$, $i=1,2,3$. There is precisely one block through the points $P_{i}$ and $Q_{i}$. We denote it by $B'_{i}$ and we denote the third point on this block by $R_{i}$, $i=1,2,3$.
  \par If the three points $Q_{1}$, $Q_{2}$ and $Q_{3}$ are contained in a common block $B'$, then this block has to be contained in $\mathcal{S}$ by the maximality condition. The only other blocks that could be contained in $\mathcal{S}$ are $B'_{1}$, $B'_{2}$ and $B'_{3}$. If all three points $R_{1}$, $R_{2}$ and $R_{3}$ are different, then only one of these blocks belongs to $\mathcal{S}$. We find an Erd\H{o}s-Ko-Rado set of size $4$ or $5$, depending on whether the block $B'$ exists. If two of the points $R_{1}$, $R_{2}$ and $R_{3}$ coincide, then we find an Erd\H{o}s-Ko-Rado set of size $5$ or $6$. If $R_{1}=R_{2}=R_{3}$, then we find an Erd\H{o}s-Ko-Rado set of size $6$ or $7$.
  \par Note that the two constructions of Erd\H{o}s-Ko-Rado sets of size $5$ give rise to isomorphic sets, so there is only one type of Erd\H{o}s-Ko-Rado sets of size $5$. Analogously, there is also only one type of Erd\H{o}s-Ko-Rado sets of size $6$.
\end{proof}

\begin{remark}
  The five types of maximal Erd\H{o}s-Ko-Rado sets in $2-(v,3,1)$ designs are explicitly described in the above theorem. Apart from the point-pencil, these block sets can be embedded in a Fano plane. However, they cannot be extended to a Fano plane by blocks of the design, due to the maximality condition. Note that the Erd\H{o}s-Ko-Rado set of size $7$ is a Fano plane that is embedded in the design.
  \par Since the four types of maximal Erd\H{o}s-Ko-Rado sets different from the point-pencil are determined by their size, we can denote them by $EKR_{i}$, $i=4,\dots,7$, the index referring to their size. Note that each of the maximal Erd\H{o}s-Ko-Rado sets different from the point-pencil, cover precisely $7$ points of the design.
\end{remark}

\begin{remark}\label{Fano}
  In a given $2-(v,3,1)$ design $\mathcal{D}$, not necessarily all five types occur. For example, if $\mathcal{D}$ is the Fano plane ($v=7$), then there is only one maximal Erd\H{o}s-Ko-Rado set, namely $EKR_{7}$, which is the set of all blocks in this case. If $\mathcal{D}$ is not a projective plane, at least two types occur, one of which is the point-pencil.
\end{remark}

We list the results for Erd\H{o}s-Ko-Rado sets on Steiner triple systems of size $v$. For small values of $v$, the results are more detailed.

\begin{theorem}\label{class3}
  Let $\mathcal{D}$ be a $2-(v,3,1)$ design.
  \begin{itemize}
    \item If $v=7$, there is only one maximal Erd\H{o}s-Ko-Rado set in $\mathcal{D}$.
    \item If $v=9$, there are two types of maximal Erd\H{o}s-Ko-Rado sets in $\mathcal{D}$, the point-pencil and $EKR_{4}$. Both contain $4$ blocks.
    \item If $v=13$, there are three types of maximal Erd\H{o}s-Ko-Rado sets in $\mathcal{D}$, the point-pencil, $EKR_{4}$ and $EKR_{5}$. The largest Erd\H{o}s-Ko-Rado sets are the point-pencils.
    \item If $v=15$, the largest Erd\H{o}s-Ko-Rado sets contain 7 blocks. There are 23 nonisomorphic $2-(15,3,1)$ designs containing an $EKR_{7}$, and 57 nonisomorphic $2-(15,3,1)$ designs not containing an $EKR_{7}$. The former have two types of maximal Erd\H{o}s-Ko-Rado sets of size $7$; for the latter all Erd\H{o}s-Ko-Rado sets of size $7$ are point-pencils.
    \item If $v\geq19$, the largest Erd\H{o}s-Ko-Rado sets are point-pencils.
  \end{itemize}
\end{theorem}
\begin{proof}
  The case $v=7$ has been treated in Remark \ref{Fano}. If $v=9$, then $\mathcal{D}$ is an affine plane of order $3$. One can see immediately that only two of the above types of maximal Erd\H{o}s-Ko-Rado sets occur, the point-pencil and the smallest one of the others, the $EKR_{4}$. Both contain four blocks. Compare this result with Remark \ref{affine}
  \par If $v=13$, there are two nonisomorphic $2-(v,3,1)$ designs. Their point sets can be denoted by $\{0,1,\dots,9,a,b,c\}$. Using \cite[Table II.1.27]{cd}, we can write the block sets as in Table \ref{v13}.
  \begin{table}[!h]
    \begin{tabularx}{\textwidth}{|X|X|X|X|X|X|X|X|X|X|X|X|X|X|X|X|X|X|X|X|X|X|X|X|X|X|}
       \hline
       0&0&0&0&0&0&1&1&1&1&1&2&2&2&2&2&3&3&3&4&4&4&5&5&5&6\\
       1&3&5&7&9&b&3&4&6&9&a&3&4&6&7&8&6&7&8&6&8&a&7&8&9&7\\
       2&4&6&8&a&c&5&7&8&b&c&9&5&a&c&b&b&a&c&c&9&b&b&a&c&9\\
       \hline
     \end{tabularx}
     
     \vspace*{0.3cm}
     
     \begin{tabularx}{\textwidth}{|X|X|X|X|X|X|X|X|X|X|X|X|X|X|X|X|X|X|X|X|X|X|X|X|X|X|}
       \hline
       0&0&0&0&0&0&1&1&1&1&1&2&2&2&2&2&3&3&3&4&4&4&5&5&5&6\\
       1&3&5&7&9&b&3&4&6&9&a&3&4&6&7&8&6&7&8&6&8&a&7&8&9&7\\
       2&4&6&8&a&c&5&7&8&b&c&9&5&a&b&c&b&c&a&c&9&b&a&b&c&9\\
       \hline
     \end{tabularx}
     \caption{Block sets}
     \label{v13}
  \end{table}
  \par We know that the point-pencil contains $6$ blocks. By Theorem \ref{mainlemma}, applied for $k=3$, $b=1$ and $r=6$, we know that any other maximal Erd\H{o}s-Ko-Rado set contains at most $5$ blocks. So, on both $2-(13,3,1)$ designs, at most three types of maximal Erd\H{o}s-Ko-Rado sets occur. Using the above notation, the two sets $\{\{0,1,2\},\{0,3,4\},\{1,3,5\},\{2,3,9\},\{2,4,5\}\}$ and $\{\{0,1,2\},\{0,3,4\},\{0,9,a\},\{2,3,9\}\}$ are maximal Erd\H{o}s-Ko-Rado sets for both $2-(13,3,1)$ designs. Hence, there are precisely three types of maximal Erd\H{o}s-Ko-Rado sets on $2-(13,3,1)$ designs.
  \par There are 80 nonisomorphic $2-(15,3,1)$ designs, see \cite[Table II.1.28]{cd} for an overview. The point-pencil contains $7$ blocks in these designs. In \cite[Table II.1.29]{cd} it is mentioned which of these 80 designs contains a Fano plane as subdesign; 23 of them do, and 57 do not. The statement follows.
  \par If $v\geq19$, then $r\geq9$, hence the point-pencil contains more blocks than the Erd\H{o}s-Ko-Rado sets of type $EKR_{i}$, $i=4,\dots,7$.
\end{proof}

Note that one of the 23 different $2-(15,3,1)$ designs having a Fano plane as subdesign, is the design consisting of the points and lines of $\PG(3,2)$. Also note that the last part of Theorem \ref{class3} is a special case of Corollary \ref{partialgeometry}.

\section{Classification results for \texorpdfstring{$k\geq4$}{k>=4}}

In this section we present the main classification theorems for Erd\H{o}s-Ko-Rado sets in $2-(v,k,1)$ designs. In Theorem \ref{rands} we will provide a proof for the result claimed in \cite{ran} about $2-(v,k,1)$ designs with large $v$. Theorem \ref{maintheorem} contains a classification theorem for $2-(v,k,1)$ designs with $v$ a little smaller. A survey result can be found in Corollary \ref{overview}.
\par In this section we will use the parameter $k_{\mathcal{S}}$, introduced in Notation \ref{nota}.

\begin{theorem}\label{rands}
  Let $\mathcal{D}$ be a $2-(v,k,1)$ design and let $\mathcal{S}$ be an Erd\H{o}s-Ko-Rado set on $\mathcal{D}$. If $r\geq k^{2}-k+1$, then $|\mathcal{S}|\leq r$. If $r=\frac{v-1}{k-1}> k^{2}-k+1$ and $|\mathcal{S}|=r$, then $\mathcal{S}$ is a point-pencil. 
\end{theorem}
\begin{proof}
  Without loss of generality, we can assume that $\mathcal{S}$ is a maximal Erd\H{o}s-Ko-Rado set. If $\mathcal{S}$ is a point-pencil, then $|\mathcal{S}|=r$. So, from now on, we can assume that $\mathcal{S}$ is not a point-pencil. By Lemma \ref{upperbounds} we know that $k_{\mathcal{S}}\leq k$. However, by the same lemma we also know that $|\mathcal{S}|\leq k^{2}-k+1$, if $k_{\mathcal{S}}\leq k$.
  \par Both statements in the theorem immediately follow.
\end{proof}

As mentioned at the end of Section \ref{sec:introduction}, there are $2-(v,k,1)$ designs with $r=k^{2}-k+1$, having a second type of Erd\H{o}s-Ko-Rado sets of size $r$.
\par Now, we look at Erd\H{o}s-Ko-Rado sets in $2-(v,k,1)$ designs with $r\leq k^{2}-k$. A classification result will be proven in Theorem \ref{maintheorem}. Before we prove some preparatory lemmas. In these lemmas we distinguish between the case $4\leq k\leq13$ and the case $k\geq14$.
\par First, we have a look at the small cases, $4\leq k\leq 13$.

\begin{table}[h]
  \centering
  \begin{tabular}{|c||c|c|c|c|c|c|c|c|c|c|}
    \hline
    $k$     & 4 & 5 & 6 & 7 & 8 & 9 & 10 & 11 & 12 & 13 \\ \hline
    $R_{k}$ & 1 & 2 & 3 & 4 & 4 & 5 & 6 & 7 & 8 & 9 \\ \hline
  \end{tabular}
  \caption{The values $R_{k}$.}
  \label{tab:4-13}
\end{table}

\begin{lemma}\label{interval4}
  Let $\mathcal{D}$ be a $2-(v,k,1)$ design, $4\leq k\leq13$, and denote $(k-1)^{2}-r=(k-1)^{2}-\frac{v-1}{k-1}$ by $R$. Let $\mathcal{S}$ be an Erd\H{o}s-Ko-Rado set on $\mathcal{D}$ with $k_{\mathcal{S}}=k-1$. If $0\leq R\leq R_{k}$, then $|\mathcal{S}|<(k-1)^{2}-R$.
\end{lemma}
\begin{proof}
  We denote $(k-1)^{2}-|\mathcal{S}|$ by $a'$, as in Lemma \ref{k-1punt}. By Lemma \ref{upperbounds} we know that $a'\geq0$. If $R<a'$, then $|\mathcal{S}|<(k-1)^{2}-R$. So, now we assume that $a'\leq R$. Since $R_{k}<k-1$, also $a'<k-1$ and we know by Lemma \ref{k-1punt} that $k(k-1)\leq|\mathcal{P}'|\leq k(k-1)+\frac{R(R-1)}{k-1-R}$. Denoting $|\mathcal{P}'|-k(k-1)$ by $b$, it follows that $0\leq b\leq \frac{R(R-1)}{k-1-R}$. By Lemma \ref{mainlemmaR} we know that
  \begin{multline*}
    |\mathcal{S}|\leq \max\left\{k^{2}-k+1-2\frac{(k^{2}-3k+1-R)(k+R)}{k(k-2)}+\frac{b(b-1)}{(k-1)(k-2)}+2\frac{(b-1)(k-1+R)}{(k-1)(k-2)},\right.\\ \left.k-1+R+\frac{R}{k-2}+\frac{b(b+k^{2}+R-2)}{k(k-2)}\right\}\;.
  \end{multline*}
  By hand or by using a computer algebra package, it can be checked that the above maximum is smaller than $(k-1)^{2}-R=r$ for all choices of $k,R,b$ fulfilling $4\leq k\leq 13$, $0\leq R\leq R_{k}$ and $0\leq b\leq \frac{R(R-1)}{k-1-R}$.
\end{proof}

Extending the calculations in the above proof, we can see that the values $R_{k}$ are optimal; enlarging one of these values leads to a contradiction.
\par Now, we look at the more general case $k\geq14$. We start with some inequalities which we will need in the proof of Lemma \ref{interval14}

\begin{lemma}\label{berekening14}
  Choose $b,c,k\in\N$, with $k\geq14$, $1\leq c\leq\frac{4}{3}k\sqrt{k}-2k-2\sqrt{k}$ and $0\leq b\leq c$. Then
  \[
    \frac{k^{3}-7k^{2}+10k-2bk-2-\sqrt{D(b,k)}}{4(k-1)}<\frac{1-c+\sqrt{(c-1)^{2}+4c(k-1)}}{2}\;,
  \]
  with $D(b,k)=(k^{3}-3k^{2}-2bk+6k-2)^{2}-8k(k-1)(b-1)(b-2)$. Furthermore, for $k\in\N$ with $k\geq14$,
  \[
    \frac{k^{3}-7k^{2}+10k-2-\sqrt{D(0,k)}}{4(k-1)}<0\;.
  \]
\end{lemma}
\begin{proof}
  First, note that $D(b,k)\geq0$ for all $0\leq b\leq \frac{4}{3}k\sqrt{k}-2k-2\sqrt{k}=:C_{k}$, hence the above functions exist.
  \par The second part of the lemma is immediate, so we focus on the first part. Note that
  \begin{align*}
    &\frac{k^{3}-7k^{2}+10k-2(b+1)k-2-\sqrt{D(b+1,k)}}{4(k-1)}-\frac{k^{3}-7k^{2}+10k-2bk-2-\sqrt{D(b,k)}}{4(k-1)}\\
    =&\frac{\sqrt{D(b,k)}-\sqrt{D(b+1,k)}-2k}{4(k-1)}\;.
  \end{align*}
  Now,
  \begin{align*}
    &\frac{\sqrt{D(b,k)}-\sqrt{D(b+1,k)}-2k}{4(k-1)}\geq0\\
    \Leftrightarrow\qquad&\sqrt{D(b,k)}-\sqrt{D(b+1,k)}\geq2k\\
    \Leftrightarrow\qquad&D(b,k)-D(b+1,k)\geq2k\left(\sqrt{D(b,k)}+\sqrt{D(b+1,k)}\right)\\
    \Leftrightarrow\qquad&2k^{3}-6k^{2}+4bk+2k-8b+4\geq\sqrt{D(b,k)}+\sqrt{D(b+1,k)}\;.
  \end{align*}
  This final inequality is valid since $\sqrt{D(b,k)}+\sqrt{D(b+1,k)}\leq2k^{3}-6k^{2}-4bk+10k-4$. These calculations show that
  \[
    \frac{k^{3}-7k^{2}+10k-2(b+1)k-2-\sqrt{D(b+1,k)}}{4(k-1)}\geq\frac{k^{3}-7k^{2}+10k-2bk-2-\sqrt{D(b,k)}}{4(k-1)}\;.
  \]
  Hence, it is sufficient to prove that
  \[
    \frac{k^{3}-7k^{2}+10k-2ck-2-\sqrt{D(c,k)}}{4(k-1)}<\frac{1-c+\sqrt{(c-1)^{2}+4c(k-1)}}{2}\;.
  \]
  Since $c\leq\frac{4}{3}k\sqrt{k}-2k-2\sqrt{k}<\frac{k^{3}-7k^{2}+8k}{2}$ for $k\geq14$, this is equivalent to
  \begin{align}
    &\left(2(k-1)\sqrt{(c-1)^{2}+4c(k-1)}+\sqrt{D(c,k)}\right)^{2}>(k^{3}-7k^{2}+8k-2c)^{2} \nonumber\\
    \Leftrightarrow\qquad &\sqrt{(c-1)^{2}+4c(k-1)}\sqrt{D(c,k)}>-2k^{4}+(9+c)k^{3}-(7c+9)k^{2}+(14c-2)k+2-6c\;.\label{k14}
  \end{align}
  Considering the left-hand side of the inequality \eqref{k14} as a function of $c$, for a fixed value of $k$, we can compute its second derivative. We find that this second derivative is negative on $\left[0,C_{k}\right]$, hence the function on the left-hand side is concave on $\left[0,C_{k}\right]$. Therefore, it dominates the function
  \[
    c\mapsto\sqrt{D(0,k)}+c\frac{\sqrt{(C_{k}-1)^{2}+4C_{k}(k-1)}\sqrt{D(C_{k},k)}-\sqrt{D(0,k)}}{C_{k}}\;.
  \]
  The slope of this line is smaller than $k^{3}-7k^{2}+14k-6$. So, we only need to check the inequality for the largest considered value for $c$, namely $C_{k}$. It turns out that this inequality is valid if $k\geq14$.
\end{proof}

In the final step of the argument we needed that $k\geq14$. This is why the cases $4\leq k\leq13$ had to be treated separately. We now discuss $2-(v,k,1)$ designs with $k_{\mathcal{S}}=k-1$. These are the hardest case in the proof of Theorem \ref{maintheorem}.

\begin{lemma}\label{interval14}
  Let $\mathcal{D}$ be a $2-(v,k,1)$ design, $k\geq14$, and denote $(k-1)^{2}-r=(k-1)^{2}-\frac{v-1}{k-1}$ by $R$. Let $\mathcal{S}$ be an Erd\H{o}s-Ko-Rado set on $\mathcal{D}$ with $k_{\mathcal{S}}=k-1$. If $0\leq R<\sqrt{k-1}$ or $\frac{1-c+\sqrt{(c-1)^{2}+4c(k-1)}}{2}\leq R<\frac{-c+\sqrt{c^{2}+4(c+1)(k-1)}}{2}$ for a value $c\in\N$, with $1\leq c\leq\frac{4}{3}k\sqrt{k}-2k-2\sqrt{k}$, then $|\mathcal{S}|<(k-1)^{2}-R$.
\end{lemma}
\begin{proof}
  Denote the interval $\left[\frac{1-c+\sqrt{(c-1)^{2}+4c(k-1)}}{2},\frac{-c+\sqrt{c^{2}+4(c+1)(k-1)}}{2}\right[$ by $I_{c}$, $c\in\N$ and $1\leq c\leq \frac{4}{3}k\sqrt{k}-2k-2\sqrt{k}:=C_{k}$, and the interval $\left[0,\sqrt{k-1}\right[$ by $I_{0}$. Recall the notation $\mathcal{P}'$. We assume that $R\in I_{c}$. From Lemma \ref{k-1punt}, it follows that $|\mathcal{P}'|\leq k(k-1)+c$. Hence, by Corollary \ref{mainlemmaR},
  \begin{multline*}
    |\mathcal{S}|\leq \max\left\{k^{2}-k+1-2\frac{(k^{2}-3k+1-R)(k+R)}{k(k-2)}+\frac{b(b-1)}{(k-1)(k-2)}+2\frac{(b-1)(k-1+R)}{(k-1)(k-2)},\right.\\\left.k-1+R+\frac{R}{k-2}+\frac{b(b+k^{2}+R-2)}{k(k-2)}\right\}\;,
  \end{multline*}
  with $b=k(k-1)-|\mathcal{P}'|$, hence $0\leq b\leq c$. Since $c\leq C_{k}$ and $R<k-2$, the inequality
  \[
    k-1+R+\frac{R}{k-2}+\frac{b(b+k^{2}+R-2)}{k(k-2)}<(k-1)^{2}-R
  \]
  clearly holds in all cases. Now, we consider the inequality
  \begin{align*}
    (k-1)^{2}-R&>k^{2}-k+1-2\frac{(k^{2}-3k+1-R)(k+R)}{k(k-2)}+\frac{b(b-1)}{(k-1)(k-2)}+2\frac{(b-1)(k-1+R)}{(k-1)(k-2)}\\
    \Leftrightarrow\qquad 0&>k+R-2\frac{(k^{2}-3k+1-R)(k+R)}{k(k-2)}+\frac{b(b-1)}{(k-1)(k-2)}+2\frac{(b-1)(k-1+R)}{(k-1)(k-2)}\;.
  \end{align*}
  This inequality is valid if and only if
  \begin{equation}\label{vglRb}
    \frac{k^{3}-7k^{2}+10k-2bk-2-\sqrt{D(b,k)}}{4(k-1)}<R<\frac{k^{3}-7k^{2}+10k-2bk-2+\sqrt{D(b,k)}}{4(k-1)}\;,
  \end{equation}
  with $D(b,k)=(k^{3}-3k^{2}-2bk+6k-2)^{2}-8k(k-1)(b-1)(b-2)$. The double inequality in \eqref{vglRb} should hold for all $b$, with $0\leq b\leq c$. Now,
  \begin{multline*}
    R<\frac{-c+\sqrt{c^{2}+4(c+1)(k-1)}}{2}\qquad\text{ and}\\\frac{k^{3}-7k^{2}+10k-2ck-2}{4(k-1)}\leq\frac{k^{3}-7k^{2}+10k-2bk-2+\sqrt{D(b,k)}}{4(k-1)}\;,
  \end{multline*}
  but the inequality $\frac{-c+\sqrt{c^{2}+4(c+1)(k-1)}}{2}<\frac{k^{3}-7k^{2}+10k-2ck-2}{4(k-1)}$ holds for all $0\leq c\leq C_{k}$ since $k\geq14$. Hence, the right inequality in \eqref{vglRb} always holds. Using
  \[
    R\geq\frac{1-c+\sqrt{(c-1)^{2}+4c(k-1)}}{2}
  \]
  and Lemma \ref{berekening14}, also the left inequality in \eqref{vglRb} follows. This finishes the proof.
\end{proof}

\begin{theorem}\label{maintheorem}
  Let $\mathcal{D}$ be a $2-(v,k,1)$ design, $k\geq4$, and let $\mathcal{S}$ be an Erd\H{o}s-Ko-Rado set on $\mathcal{D}$. If $k^{2}-k\geq r=\frac{v-1}{k-1}\geq k^{2}-3k+\frac{3}{4}\sqrt{k}+2$, then $|\mathcal{S}|\leq r$. If $(r,k)\neq(8,4)$, equality is obtained if and only if $\mathcal{S}$ is a point-pencil. 
\end{theorem}
\begin{proof}
  Without loss of generality, we can assume that $\mathcal{S}$ is a maximal Erd\H{o}s-Ko-Rado set. Recall the notation $k_{\mathcal{S}}$. If $\mathcal{S}$ is a point-pencil, then $|\mathcal{S}|=r$. So, from now on, we can assume that $\mathcal{S}$ is not a point-pencil. By Lemma \ref{upperbounds} we know that $k_{\mathcal{S}}\leq k$. We distinguish between three cases.
  \begin{itemize}
    \item If $k_{\mathcal{S}}=k-1$, then $|\mathcal{S}|\leq k^2-2k+1$ by Lemma \ref{upperbounds}. In this case, if $k^{2}-2k+1<r\leq k^{2}-k$, the theorem clearly holds, so we assume $r\leq k^{2}-2k+1$. As before, we denote $R=(k-1)^{2}-r$. First, assume that $k\geq14$. In this case, $0\leq R\leq k-\frac{3}{4}\sqrt{k}-1$. So, $0\leq R<\sqrt{k-1}$ or there is a value $c\in\N$, with $1\leq c\leq \frac{4}{3}k\sqrt{k}-2k-2\sqrt{k}$, such that $\frac{1-c+\sqrt{(c-1)^{2}+4c(k-1)}}{2}\leq R<\frac{-c+\sqrt{c^{2}+4(c+1)(k-1)}}{2}$. Applying Lemma \ref{interval14} we find that $|\mathcal{S}|<(k-1)^{2}-R=r$.
    \par Now assume that $4\leq k\leq 13$. In this case, $0\leq R\leq R_{k}=\left\lfloor k-\frac{3}{4}\sqrt{k}-1\right\rfloor$. Applying Lemma \ref{interval4}, we find that $|\mathcal{S}|<(k-1)^{2}-R=r$.
    \item If $k_{\mathcal{S}}=k$, then $|\mathcal{P}'|=k^{2}-k+1$ by Lemma \ref{kpunt}. So, we can apply Lemma \ref{mainlemma} with $b=1$. We find that
    \[
      |\mathcal{S}|\leq\max\left\{k^{2}-k+1-\frac{2(r-k)(k^{2}-k+1-r)}{k(k-2)},k^{2}-r-\frac{r-1}{k-2}+\frac{2k(k-1)-r}{k(k-2)}\right\}\;.
    \]
    The inequality $k^{2}-k+1-\frac{2(r-k)(k^{2}-k+1-r)}{k(k-2)}<r$ holds if and only if $\frac{k^{2}}{2}<r<k^{2}-k+1$. If $k\geq5$, this condition is fulfilled since $k^{2}-k+1>k^{2}-k$ and $\frac{k^{2}}{2}<k^{2}-3k+\frac{3}{4}\sqrt{k}+2$. If $k=4$ and $R=0$, hence $r=9$, then $k^{2}-k+1-\frac{2(r-k)(k^{2}-k+1-r)}{k(k-2)}=8<r$; if $k=4$ and $R=1$, hence $r=8$, then $k^{2}-k+1-\frac{2(r-k)(k^{2}-k+1-r)}{k(k-2)}=8=r$.
    \par Since $k^{2}-3k+\frac{3}{4}\sqrt{k}+2>\frac{k^{2}}{2}-\frac{k}{4}+\frac{3}{8}$ for all $k\geq4$, the inequality $k^{2}-r-\frac{r-1}{k-2}+\frac{2k(k-1)-r}{k(k-2)}<r$ is fulfilled in all cases.
    \item If $k_{\mathcal{S}}\leq k-2$, then $|\mathcal{S}|\leq k^2-3k+1$ by Lemma \ref{upperbounds}. Clearly, $k^2-3k+1<k^{2}-3k+\frac{3}{4}\sqrt{k}+2\leq r$.
  \end{itemize}
  Hence, for $k\geq5$, in all three cases $|\mathcal{S}|<r$; for $k=4$, in all three cases $|\mathcal{S}|\leq r$ and moreover $|\mathcal{S}|< r$ if $r\neq8$. The theorem follows.
\end{proof}

We now summarize the results of this section.

\begin{corollary}\label{overview}
  Let $\mathcal{D}$ be a $2-(v,k,1)$ design, $k\geq4$, with $r=\frac{v-1}{k-1}\geq k^{2}-3k+\frac{3}{4}\sqrt{k}+2$, and let $\mathcal{S}$ be an Erd\H{o}s-Ko-Rado set on $\mathcal{D}$. Then $|\mathcal{S}|\leq r$. If $r\neq k^{2}-k+1$ and $(r,k)\neq(8,4)$, then $|\mathcal{S}|=r$ if and only if $\mathcal{S}$ is a point-pencil. 
\end{corollary}
\begin{proof}
  This follows immediately from Theorem \ref{rands} and Theorem \ref{maintheorem}.
\end{proof}

\section{Maximal \texorpdfstring{Erd\H{o}s}{Erdos}-Ko-Rado sets in unitals}

The results from Lemma \ref{upperbounds}, Lemma \ref{mainlemma}, Lemma \ref{kpunt} and Lemma \ref{k-1punt} can also be used in a different way. For a fixed class of designs, with $v$ (or equivalently $r$) a function of $k$, an upper bound on the size of the largest maximal Erd\H{o}s-Ko-Rado set different from a point-pencil can be computed. We show this for the unitals. Recall that a $2-(q^{3}+1,q+1,1)$ design is a unital of order $q$. First we state Lemma \ref{mainlemma} for a unital of order $q$.

\begin{lemma}\label{mainlemmaU}
  Let $\mathcal{U}$ be a unital of order $q$ and let $\mathcal{S}$ be an Erd\H{o}s-Ko-Rado set on $\mathcal{U}$ such that $|\mathcal{P}'|=q(q+1)+b$, whereby $\mathcal{P}'$ is the set of points covered by the elements of $\mathcal{S}$. Then
  \[
    |\mathcal{S}|\leq \max\left\{q^{2}-q+1+\frac{b(b-1)}{q(q-1)}+\frac{2b}{q-1},q+\frac{bq(q+2)}{q^{2}-1}+\frac{b(b-1)}{q^{2}-1}\right\}\;.
  \]
\end{lemma}

\begin{lemma}\label{q+1puntgrens}
  Let $\mathcal{U}$ be a unital of order $q$ and let $\mathcal{S}$ be an Erd\H{o}s-Ko-Rado set on $\mathcal{U}$   with $k_{\mathcal{S}}=q+1$. If $q\geq4$, then $|\mathcal{S}|\leq q^{2}-q+1$. If $q=3$, then $|\mathcal{S}|\leq8$.
\end{lemma}
\begin{proof}
  By Lemma \ref{kpunt} we know that $|\mathcal{P}'|=q^{2}+q+1$. We apply Lemma \ref{mainlemmaU} and we find that $|\mathcal{S}|\leq\max\left\{q^{2}-q+1+\frac{2}{q-1},q+\frac{q(q+2)}{q^{2}-1}\right\}$. The lemma immediately follows.
\end{proof}

\begin{lemma}\label{qpuntgrens}
  Let $\mathcal{U}$ be a unital of order $q$ and let $\mathcal{S}$ be an Erd\H{o}s-Ko-Rado set on $\mathcal{U}$ 
  with $k_{\mathcal{S}}=q$. If $q\geq5$, then $|\mathcal{S}|\leq q^{2}-q+\sqrt[3]{q^{2}}-\frac{2}{3}\sqrt[3]{q}+1$. If $q=3$, then $|\mathcal{S}|\leq7$; if $q=4$, then $|\mathcal{S}|\leq13$.
\end{lemma}
\begin{proof}
  Denote $q^{2}-|\mathcal{S}|$ by $a'$. We can assume $a'<q$ since otherwise the lemma clearly holds. By Lemma \ref{upperbounds}, we know that $a'\geq0$, and by Lemma \ref{k-1punt} we know that $|\mathcal{P}'|=q^{2}+q+b$, with $0\leq b\leq \frac{a'^{2}-a'}{q-a'}$. We apply Lemma \ref{mainlemmaU} and we find that
    \begin{multline*}
    |\mathcal{S}|\leq q^{2}-q+1+2\frac{a'(a'-1)}{(q-a')(q-1)}+\frac{a'(a'-1)(a'^{2}-q)}{q(q-1)(q-a')^{2}}\\\text{or }\quad|\mathcal{S}|\leq q+\frac{qa'(q+2)(a'-1)}{(q^{2}-1)(q-a')}+\frac{a'(a'-1)(a'^{2}-q)}{(q-a')^{2}(q^{2}-1)}\;.
  \end{multline*}
  Using $|\mathcal{S}|=q^{2}-a'$, the first inequality can be rewritten as
  \[
    q(q-a'-1)(q-a')^{2}(q-1)\leq a'(a'-1)(2q^{2}-2qa'+a'^{2}-q)\;.
  \]
  For $q=3$, this implies $a'\geq 2$ and for $q=4$, this implies $a'\geq 3$. For general $q$, it implies $a'\geq q-\sqrt[3]{q^{2}}+\frac{2}{3}\sqrt[3]{q}-1$.
  \par Now we look at the second inequality. Using $|\mathcal{S}|=q^{2}-a'$, it can be rewritten as
  \[
    (q^{2}-q-a')(q-a')^{2}(q^{2}-1)\leq a'(a'-1)(q^{3}-(a'-2)q^{2}-(2a'+1)q+a'^{2})\;.
  \]
  Using that $0\leq a'<q$, it follows that $a'=q-1$.
  \par Only one of the inequalities needs to hold, but $q-\sqrt[3]{q^{2}}+\frac{2}{3}\sqrt[3]{q}-1\leq q-1$. The lemma follows.
\end{proof}

\begin{theorem}\label{classnonclass}
  Let $\mathcal{U}$ be a unital of order $q$ and let $\mathcal{S}$ be a maximal Erd\H{o}s-Ko-Rado set on $\mathcal{U}$. If $q\geq5$, then either $|\mathcal{S}|=q^{2}$ and $\mathcal{S}$ is a point-pencil, or else $|\mathcal{S}|\leq q^{2}-q+\sqrt[3]{q^{2}}-\frac{2}{3}\sqrt[3]{q}+1$. If $q=4$, then either $|\mathcal{S}|=16=q^{2}$ and $\mathcal{S}$ is a point-pencil, or else $|\mathcal{S}|\leq 13=q^{2}-q+1$. If $q=3$, then either $|\mathcal{S}|=9=q^{2}$ and $\mathcal{S}$ is a point-pencil, or else $|\mathcal{S}|\leq 8$.
\end{theorem}
\begin{proof}
  If $\mathcal{S}$ is a point-pencil, then it contains $q^{2}$ elements. From now on, we assume that $\mathcal{S}$ is not a point-pencil. Recall the definition of $k_{\mathcal{S}}$. By Lemma \ref{upperbounds}, $k_{\mathcal{S}}\leq q+1$. Moreover, if $k_{\mathcal{S}}\leq q-1$, then $|\mathcal{S}|\leq q^{2}-q-1$.
  \par First, we assume $q\geq5$. If $k_{\mathcal{S}}=q$, then $|\mathcal{S}|\leq q^{2}-q+\sqrt[3]{q^{2}}-\frac{2}{3}\sqrt[3]{q}+1$ by Lemma \ref{qpuntgrens}. If $k_{\mathcal{S}}=q+1$, then $|\mathcal{S}|\leq q^{2}-q+1$ by Lemma \ref{q+1puntgrens}.
  \par The results for $q=3,4$ are obtained in the same way, using the results from Lemma \ref{q+1puntgrens} and Lemma \ref{qpuntgrens}.
\end{proof}

\begin{remark}
  Note that these results correspond with the result for classical unitals in Corollary \ref{classclass} since the triangle contains only $q+2$ blocks.
  \par Note that the unitals are not covered by Corollary \ref{partialgeometry}. However, they are covered by Theorem \ref{overview}. So we already knew that the point-pencils are the largest Erd\H{o}s-Ko-Rado sets. The above theorem thus gives a bound on the size of the second-largest maximal Erd\H{o}s-Ko-Rado set.
\end{remark}

\paragraph*{Acknowledgement:}The author wants to thank the anonymous referees for improving the quality of the article. The research of the author is supported by FWO-Vlaanderen (Research Foundation - Flanders).

\end{document}